\documentclass[12pt]{amsart}
\usepackage{amscd}
\usepackage{amssymb}
\usepackage{a4wide}
\usepackage{amstext}
\usepackage{amsthm}
\usepackage{mathrsfs}
\usepackage{xcolor}

\renewcommand{\phi}{\varphi}
\newcommand{\al}{\alpha}
\newcommand{\co}{\mathbb{C}}
\newcommand{\N}{\mathbb{N}}
\newcommand{\Z}{\mathbb{Z}}

\newcommand{\R}{\mathbb{R}}
\newcommand{\rea}{{\rm Re}\,}
\newcommand{\ima}{{\rm Im}\,}

\newtheorem{Thm}{Theorem}[section]
\newtheorem{theorem}[Thm]{Theorem}
\newtheorem{lemma}[Thm]{Lemma}
\newtheorem{proposition}[Thm]{Proposition}
\newtheorem{corollary}[Thm]{Corollary}
\newtheorem{remark}[Thm]{Remark}

\DeclareMathOperator{\diag}{diag}

\begin{document}  
\sloppy

\title[Systems biorthogonal to exponential systems]
{Systems biorthogonal to exponential systems \\ on a finite union of intervals}
\author{Anton Baranov, Yurii Belov, Alexander Kuznetsov}

\address{
 Anton Baranov
\newline Department of Mathematics and Mechanics, St.~Petersburg State University,
St.~Petersburg, Russia
\newline {\tt anton.d.baranov@gmail.com}
\smallskip
\newline \phantom{x}\,\, Yurii Belov
\newline  Department of Mathematics and Computer Science, St.~Petersburg State
University, St. Petersburg, Russia
\newline {\tt j\_b\_juri\_belov@mail.ru}
\smallskip
\newline \phantom{x}\,\, Alexander Kuznetsov
\newline  Department of Mathematics and Computer Science, St.~Petersburg State
University, St. Petersburg, Russia
\newline {\tt alkuzn1998@gmail.com}
}

\thanks{The results of Sections 2 and 3 were obtained with the support of Russian Science Foundation grant 19-71-30002. The results of Sections 4 and 5  were obtained with the support of Ministry of Science and Higher Education of the Russian Federation, agreement No 075-15-2021-602.}

\begin{abstract}
We study the properties of a system biorthogonal to a complete and minimal system of exponentials in $L^2(E)$,
where $E$ is a finite union of intervals, and show that 
in the case when $E$ is a union of two or three intervals the biorthogonal system is also complete. 
\end{abstract}

\maketitle

\section{Introduction}

Geometric properties of exponential systems on an interval are among the major themes of 20th century harmonic analysis. 
This theory emerged in the classical works of R.~Paley and N.~Wiener, N.~Levinson, 
R.J.~Duffin and A.C.~Schaeffer;
it includes such milestones as the work of A.~Beurling and P.~Malliavin on the radius of completeness \cite{bm}, 
solution of the exponential Riesz bases problem (which goes back to Paley and Wiener) by B.S.~Pavlov, 
S.V.~Khruschev and N.K.~Nikolski (see \cite{pav, nik, min}), or more recent description
of exponential frames by K.~Seip and J.~Ortega-Cerd\`a \cite{os}.
Main tools for the study of exponential systems were provided by the entire function theory: an application 
of the Fourier transform allowed to work with equivalent problems (like uniqueness, sampling or interpolation)
in the Paley--Wiener space of bandlimited functions. 

More recently, the study of exponential systems on more general disconnected sets became 
a field of intensive research. In this case many of the above-mentioned problems become much more complicated; 
not only the description but mere the existence of exponential Riesz bases with real frequencies 
on a union of three or more intervals was an open problem which was only recently solved by 
G.~Kozma and S. Nitzan \cite{kn}. Even more recent is a construction of a bounded set which does not admit an exponential Riesz basis by G.~Kozma, S.~Nitzan and
A.~Olevskii \cite{kno}.
We refer to the monograph \cite{ol} by A.~Olevskii and A.~Ulanovskii for many recent advances in the field. 
One of the difficulties compared to the case of one interval is that the associated spaces of entire functions
become much more involved. Therefore, most of the proofs are based on real analysis methods. 

\subsection{}
In the present paper we are interested in the following problem. Let $E\subset \R$ be a bounded set and 
assume that a system of exponentials $\{e_\lambda\}_{\lambda\in \Lambda}$, where $e_\lambda(t) = e^{i\lambda t}$
and $\Lambda \subset \co$,
is complete and minimal in $L^2(E)$. Then there exists a unique biorthogonal system, i.e., 
a system  $\{g_\lambda\}_{\lambda\in\Lambda} \subset L^2(E)$ such that
$$
(e_\lambda, g_\mu)_{L^2(E)} =
\begin{cases} 1, & \lambda = \mu, \\
0, & \lambda \ne \mu.
\end{cases}
$$
Is it true that the  biorthogonal system $\{g_\lambda\}_{\lambda\in\Lambda}$ is also complete? This is a natural question since completeness of the biorthogonal system means that there is a one-to-one correspondence
between functions $f\in L^2(E)$  and their generalized (non-harmonic) Fourier series 
$\sum_{\lambda\in \Lambda} (f, g_\lambda) e_\lambda$. 

It is a result due to R.M.~Young \cite{young}
that a system biorthogonal to a complete and minimal system of exponentials in $L^2(I)$,
where $I$ is an interval, is complete. In \cite{belov2015} the second author showed that
a complete and minimal system of time-frequency shifts of the Gaussian (a Gabor system) in $L^2(\R)$ 
always has a complete biorthogonal system. This result is obtained by passing, via the Bargmann transform, to an equivalent
problem for systems of reproducing kernels in the Bargmann--Fock space. 

Completeness of systems biorthogonal to systems of reproducing kernels in different spaces of analytic functions was studied in 
\cite{fric, bb2011, bbb2018}. In particular, in \cite{bbb2018} the completeness of the biorthogonal system was proved 
for a wide class of weighted Fock-type spaces under very mild regularity conditions on the weight. 
In \cite{bb2011} completeness of the biorthogonal system was studied
in de Branges spaces and model subspaces of the Hardy space. In this paper 
the very first examples were produced of a space of analytic functions where 
a system biorthogonal to  a system of reproducing kernels can be incomplete and even can have an arbitrary finite or infinite
defect.

\subsection{}
The main result of the present paper says that in the case when $E$ is a union of two or three intervals 
Young's theorem remains true: 
a system biorthogonal to a complete and minimal system of exponentials in $L^2(E)$ is complete.

Let $\{e_\lambda\}_{\lambda\in\Lambda}$ be complete and minimal in $L^2(E)$. In the case when
$E$ is an interval it is well known that $\Lambda$ must satisfy
\begin{equation}
\label{dens}
\mathcal{D}_+(\Lambda) =\limsup_{R\rightarrow\infty}\frac{\# (\Lambda\cap B(0,R))}{2R}=\frac{|E|}{2\pi}.
\end{equation}
Here and in what follows $|E|$ will denote the Lebesgue measure of $E$ and $B(0,R)$ is a disk of radius $R$ centered at $0$.
For an arbitrary $E$ there are well known Landau estimates of the uniform densities for the case
when the system $\{e_\lambda\}_{\lambda\in\Lambda}$ is a frame or a Riesz sequence in $L^2(E)$.
Moreover, A. Olevskii and A. Ulanovskii \cite{ou1} showed  that in the case when $E$ is an arbitrary compact set, 
$\Lambda\subset \R$
and $\{e_\lambda\}_{\lambda\in\Lambda}$ is complete and uniformly minimal in $L^2(E)$, one has 
$$
\mathcal{D}^u_+(\Lambda) =\limsup_{R\rightarrow\infty}\frac{ \sup_{x\in \R}  \# (\Lambda\cap (x,x+R))}{R}\le\frac{|E|}{2\pi},
$$
where $\mathcal{D}^u_+(\Lambda)$ is the so-called Beurling upper uniform density.

We show (see Corollary \ref{upp}) that when $E$ is a finite union of intervals 
and the system $\{e_\lambda\}_{\lambda\in\Lambda}$ is complete and minimal,
one  has
$\mathcal{D}_+(\Lambda) \le \frac{|E|}{2\pi}$. We do not know whether the lower estimate in 
\eqref{dens} always holds and we impose it as an
additional condition. Then our main theorem reads as follows.

\begin{theorem}
\label{main}
Let $E$ be a union of two or three intervals, let $\Lambda\subset \co$ 
and let $\{e_\lambda\}_{\lambda\in\Lambda}$ be a complete and minimal system in $L^2(E)$ satisfying
\eqref{dens}. Then the system biorthogonal to $\{e_\lambda\}_{\lambda\in\Lambda}$ is also complete.
\end{theorem} 

It seems to be an interesting (and, apparently, complicated) problem whether
 a complete and minimal system of exponentials in $L^2(E)$ always must have the maximal density $|E|/(2\pi)$.
It is known that in the setting of Gabor systems of Gaussians the corresponding upper density (measured with respect to
the area Lebesgue measure) of a complete and minimal Gabor system can change in the range from $\pi^{-1}$ to $e$
\cite{bbk}. 
Also, it was noted already by H. Landau \cite{lan} (see, also, \cite[Section 6.2]{ol}) that for a finite 
union of intervals the density of a complete system (but without the minimality condition) can 
be arbitrarily small when compared with $|E|$.
\medskip
\\
{\bf Conjecture.} {\it Let $E$ be a finite union of intervals and let 
$\{e_\lambda\}_{\lambda\in\Lambda}$ be complete and minimal in $L^2(E)$.
Then $\Lambda$ satisfies \eqref{dens}. }
\medskip 

In contrast to most of results on disconnected sets, the methods of the proof of Theorem \ref{main}
are complex analytic. For this we will need to obtain a formula for a system biorthogonal to 
a system of reproducing kernels in $PW_E$, the Paley--Wiener space on disconected spectrum. Also, in the case of three intervals we 
will essentially use a construction of Riesz basis of exponentials in $L^2(E)$ with some additional properties
which is based on the methods of the paper \cite{kn} by G.~Kozma and S.~Nitzan. However, for the moment our results
do not extend to a larger number of intervals and the following remains an open problem.
\medskip
\\
{\bf Problem.} {\it Let $E$ be a union of at least four disjoint intervals \textup(or simply a bounded measurable set\textup)
and let 
$\{e_\lambda\}_{\lambda\in\Lambda}$ be complete and minimal in $L^2(E)$.
Is it true that its biorthogonal system is complete? }
\medskip 

\subsection{Notation and organization of the paper} Throughout the paper we write
$A\asymp B$ if $C_1A\leq B\leq C_2 B$ for some positive constants $C_1$ and $C_2$
and for all admissible values of the parameters.
 
The paper is organized as follows. In Section \ref{prelim} we prove some preliminary results and, in particular, give a construction
of a Riesz basis on three intervals with some additional properties. In Section \ref{bio} formulas
for a system biorthogonal to a system of reproducing kernels of the Paley--Wiener space
on a disjoint spectrum are obtained. Theorem \ref{main} for the case of two intervals
is proved in Section \ref{case2}; this proof is more elementary and, in particular, does not use the results of \cite{kn}.
Finally, in Section \ref{case3} we prove Theorem \ref{main} for three intervals.


\section{Preliminaries}
\label{prelim}

As usual we will pass via Fourier transform to an equivalent problem in the associated space of entire functions. 
Given a bounded set $E\subset \R$, denote by $PW_E$ the space  of all entire functions $f$
representable as 
$$
f(z) = ({\mathcal F}\phi)(z) = \int_E \phi(t) e^{itz} dt,
$$
where $\phi\in L^2(E)$. Clearly, $PW_E \subset L^2(\R)$ and $\|f\|_2^2 = 2\pi \|\phi\|_2^2$. 

Denote by $k^E_\lambda$ the reproducing kernel of $PW_E$
at the point $\lambda$. It is clear that 
$$
k^E_\lambda (z) = \frac{1}{2\pi}{\mathcal F} (e^{-it \bar \lambda})(z) =
\frac{1}{2\pi} \int_E e^{it(z-\bar \lambda)} dt.
$$ 
Thus, a system of exponentials in $L^2(E)$ is mapped to a system of reproducing kernels in $PW_E$, and 
we come to an equivalent problem: {\it given a complete and minimal system of reproducing kernels in 
$PW_E$, is it true that its biorthogonal is also complete? }

One of the key difficulties while working with $PW_E$ is that this space is not division-invariant unless $E$ 
is an interval, i.e., for  $f\in PW_E$ the equality $f(\lambda) =0 $ does not imply that 
$\frac{f(z)}{z-\lambda} \in PW_E$.
It is well known that in a  division-invariant Hilbert space of analytic functions a system biorthogonal to a 
complete and minimal system of reproducing kernels is given by $\frac{G(z)}{G'(\lambda)(z-\lambda)}$ where
$G$ is some fixed {\it generating} function with zeros at $\Lambda$. This is no longer true in 
$PW_E$.

The following lemma gives a representation of
$\frac{f(z)}{z-\lambda}$ as a sum of a function from $PW_E$ and 
of a linear combination of reproducing kernels in the complementing intervals.  

\begin{lemma} 
\label{div}
Let $E=\cup_{j=1}^n I^j$, where $I^j =[a_j, b_j]$ are disjoint intervals 
with $b_j <a_{j+1}$, and let $L^j = [b_j, a_{j+1}]$ be the complementing intervals. 
If $f\in PW_E$ and $f(\lambda)=0$, then 
$$
\frac{f(z)}{z-\lambda}=\tilde{f}(z)+ \sum_{j=1}^{n-1} c_j k^{L^j}_{\bar \lambda} (z)
$$
where $\tilde{f}$ is some function from $ PW_E$ and  
$c_j = - 2\pi i F^j (\lambda)$, where $F^j$ is the projection of $f$ onto 
$PW_{\cup_{k=1}^j I^j}$.
\end{lemma}

\begin{proof} Without loss of generality we can assume that $a_1= 0$ and $b_n = 2\pi$. We have 
$$
f(z)=\int_E\varphi(t)e^{itz}dt,\quad \varphi\in L^2(E), \qquad f(\lambda)=\int_E \varphi(t)e^{i t\lambda}dt=0.
$$
We consider $\varphi$ as an element of $L^2(0, 2\pi)$, $\varphi \equiv 0$ on $(0,2\pi) \setminus E$. 
Let $\Phi(x) = \int_0^x \varphi(t)e^{i t\lambda } dt$ be the primitive of $\varphi(t)e^{i t \lambda}$ such that $\Phi(0)=0$.
Hence, $\Phi(2\pi)=\int_{0}^{2\pi}\varphi(t)e^{i t\lambda}dt=0$. Integrating by parts, we get
$$
f(z)=\int_{-\pi}^\pi e^{it(z-\lambda)}d\Phi(t)= -i(z-\lambda)\int_{-\pi}^\pi e^{it(z-\lambda)}\Phi(t)dt.
$$
Note that $\Phi$ is a constant on each of the intervals $L^j$. So,
$$
\frac{f(z)}{z-\lambda}=- i\int_E\Phi(t)e^{it(z-\lambda)}dt -i
\sum_{j=1}^{n-1} \Phi(b_j) \int_{L^j} e^{i t(z-\lambda)}dt=\tilde{f}(z)  - 2\pi i \Phi(b_j) k^{L^j}_{\bar \lambda} (z),
$$
where $\tilde{f}\in PW_E$.

Finally, note that  $\Phi (b_j) = \int_0^{b_j} \varphi(t)e^{i t\lambda } dt = F^j (\lambda)$,
where $F^j(z) = \int_{\cup_{k=1}^j I_k}  \varphi(t)e^{i tz}dt$ is the projection of $f$ onto 
$PW_{\cup_{k=1}^j I^j}$.
\end{proof}

\begin{corollary}
\label{div1}
Let $E=\cup_{j=1}^n I^j$, where $I^j$ are disjoint intervals and let
$f\in PW_E$, $f = \sum_{j=1}^n f^j$, $f^j \in PW_{I^j}$.
Assume that $f(\lambda) =0$. Then $\frac{f(z)}{z-\lambda} \in PW_E$ if and only if
$f^j(\lambda) = 0$ for any $j$.
\end{corollary}

\begin{proof}
By Lemma \ref{div}, 
 $\frac{f(z)}{z-\lambda} \in PW_E$ if and only if
$F^j(\lambda) = 0$ for any $j$. It remains to note that $F^j = \sum_{k=1}^j f^k$.
\end{proof}

\begin{remark}
\label{div2}
{\rm We will sometimes use this statement in a slightly more general setting. Assume that $f\in PW_E +z PW_E$.
Then $f$ is not in the Paley--Wiener space, but it still can be written as $f= \sum_{j=1}^n f^j$
where the functions $f^j$ have their spectra (understood in the distributional sense) in $I^j$.
Indeed, if $g=\sum_{j=1}^n g^j \in PW_E$ with $g_j\in PW_{I^j}$, then 
$zg=\sum_{j=1}^n z g^j$ and the spectrum of $zg^j$ is contained in $I^j$. 

It follows easily from Corollary \ref{div1} that for $f\in PW_E +z PW_E$ one still has
$\frac{f(z)}{z-\lambda} \in PW_E$ if and only if $f^j(\lambda) = 0$ for any $j$.  }
\end{remark}

As mentioned above, in the case when $E$ is a finite union of intervals Riesz bases of exponentials with real 
frequencies were constructed G. Kozma and S. Nitzan \cite{kn}. We use their method to construct an exponential basis 
on three intervals with some additional properties. 

In what follows we will often use the notion of the (conjugate) indicator diagram of an entire function of exponential type.
Recall that for a function in $PW_I$, where $I = [a,b]$ is an interval, its indicator diagram is contained in $-iI = [-ib,  -ia]$.
The same is true for a function which is in $\mathcal{P}\cdot PW_I$ (where $\mathcal{P}$ denotes the set of all polynomials)
and thus has its spectrum in $I$ in the distributional sense. It will be more convenient for us
to work with the conjugate indicator diagrams since for $f\in PW_I$ its conjugate indicator diagram
is contained in $iI$. We refer to \cite[Lecture 9]{lev} or \cite[Chapter 5]{boas} 
for the definition and the properties of indicator diagrams. We will denote by $\diag f$
the conjugate indicator diagram of an entire function $f$ of exponential type.

\begin{proposition}
\label{basis}
Let $E = [0, a] \cup [b,c] \cup [d, 2\pi]$. Then there exists a sequence $\Gamma = \{\gamma_n\} \subset \mathbb{Z}$ 
such that $\{e_\gamma\}_{\gamma\in \Gamma}$ is a Riesz basis in $L^2(E)$ and, moreover, there exists an entire function $G$ 
with the conjugate indicator diagram $[0, i|E|]$ such that its zero set coincides with $\Gamma$ and 
$$
|G(z)| \asymp {\rm dist}\,(z, \Gamma), \qquad |\ima z| \le 1.
$$
In particular, $|G'(\gamma)| \asymp 1$, $\gamma \in \Gamma$.
\end{proposition}

We start with the case of one interval. The proof follows a nice idea due to K. Seip \cite{seip}.

\begin{lemma}
\label{bas1}
Let $\al\in(0,1)$. Then there exists $\Gamma = \{\gamma_n\} \subset \alpha\Z$ with $\gamma_0 = 0$ 
such that $\{e_\gamma\}_{\gamma \in \Gamma}$ is a Riesz basis for $L^2(0,2\pi)$ and the function
\begin{equation}
\label{vp}
G(z) = e^{\pi i z} \lim_{R\to\infty}  z \prod_{0<|\gamma_n| <R}\Big(1-\frac{z}{\gamma_n}\Big) 
= z e^{\pi i z} \, {\rm v.p.}\prod_{n\ne 0}\Big(1-\frac{z}{\gamma_n}\Big)
\end{equation}
has conjugate indicator diagram $[0, 2\pi i]$ and satisfies $|G(z)| \asymp  {\rm dist}\,(z, \Gamma)$, $|\ima z| \le 1$.
\end{lemma}

\begin{proof}
We will choose $\Gamma$ as a small perturbation  of integers.  
Fix a sufficiently large $M\in \mathbb{N}$ so that $M>4$ and 
\begin{equation}
\label{alp}
(1-\al)M^2 >2M+2.
\end{equation}
Now consider the interval $(0, M]$ and choose the points $\gamma_1, \gamma_2, \dots \gamma_M\in (0,M]\cap\al\Z$, 
$\gamma_1< \gamma_2< \dots<\gamma_M$,
so that for the differences $\delta_k = \gamma_k- k$ we have
\begin{equation}
\label{gor}
\Delta_1 = \sum_{k=1}^M \delta_k \in (-\alpha, \alpha).
\end{equation}
Let us show that such choice is possible. If we take $\gamma_k = \alpha k$, $k=1, \dots, M$ 
(i.e., the smallest possible choice), then we have
$$
\Delta_1 =  \sum_{k=1}^M (\alpha-1)k = \frac{(\alpha-1)M(M+1)}{2} <-1.
$$
On the other hand, if we take the largest possible $\gamma_k$, 
namely $\gamma_k = \alpha \big[\frac{M}{\al}\big] +\al(k -M)$, then
$$
\Delta_1 = \al M \bigg[\frac{M}{\al}\bigg] -\al M^2 - \frac{(1-\alpha)M(M+1)}{2} \ge
\frac{(1-\alpha)M^2}{2} -\frac{(1+\alpha)M}{2} >1.
$$
Now let us start moving the points from the group $\{\alpha k, \ k=1, \dots M\}$ to the right starting from the largest point. Since changing 
the value of  $\gamma_k$ from $\alpha m$ to $\alpha (m+1)$ adds  $\alpha$ to $\Delta_1$ we can achieve \eqref{gor} for some choice of 
distinct $\gamma_k$. 

Now we repeat the procedure and choose 
$\gamma_{M+1},\gamma_{M+2}, \dots \gamma_{2M}\in (M,2M]\cap\al\Z$, $\gamma_{M+1} <\gamma_{M+2}
 \dots<\gamma_{2M}$,
so that 
$$
\Delta_2 = \sum_{k=M+1}^{2M} \delta_k \in (-\alpha, \alpha), \qquad \delta_ k = \gamma_k-k.
$$
Moreover, we can also choose the points so that $\Delta_1$ and $\Delta_2$ are of different signs and so 
$\Delta_1   + \Delta_2 \in (-\alpha, \alpha)$. Analogously, for any interval $(lM+1, (l+1)M]$, $l\in \mathbb{Z}$,
we find the points $\gamma_{lM+1},\gamma_{lM+2}, \dots \gamma_{lM +M}\in (lM+1, (l+1)M]\cap\al\Z$ 
so that 
$$
\Delta_l= \sum_{k=lM+1}^{(l+1)M} \delta_k \in (-\alpha, \alpha), \qquad \sum_{k=l_1}^{l_2} \Delta_l \in (-\alpha, \alpha)
$$
for any $l, l_1, l_2\in\Z$. Note also that, by the construction, $\{\delta_k\}\in \ell^\infty$. Thus,
\begin{equation}
\label{ab}
\sup_{m_1, m_2} \bigg| \sum_{k=m_1}^{m_2} \delta_k\bigg| <\infty.
\end{equation}
Without loss of generality we can take $\gamma_0 = 0$. 

Since
$$
\frac{1}{M} \bigg| \sum_{k=lM+1}^{(l+1)M} \delta_k \bigg|  <\frac{1}M <\frac{1}{4},
$$
the sequence $\Gamma = \{\gamma_n\}_{n\in\Z}$ generates a Riesz basis of exponentials
in $L^2(0, 2\pi)$ by the Avdonin theorem \cite{avd, nik}.

We define $G$ by \eqref{vp} (the product converges, since $\delta_k$ are bounded).
Fix some $z$ with $|\ima z| \le 1$ and sufficiently large $\rea z>0$. 
Let $\gamma_m$ be the element of $\Gamma$ closest to $z$ and
 $l$ be the integer closest to $z$. Let us show  that 
$$
|G(z)| \asymp |z-\gamma_m| \prod_{n\in\Z\setminus\{0,m\}} \bigg|1-\frac{\gamma_m}{\gamma_n} \bigg| \asymp
|z-\gamma_m| \prod_{n\in\Z\setminus \{0, l\}} \bigg|1-\frac{z}{n} \bigg|
\asymp |z-\gamma_m| \frac{|\sin \pi z|}{|z-l|}  \asymp |z-\gamma_m|.
$$
Indeed, 
$$
\begin{aligned}
\log  \prod_{n\in\Z\setminus\{0,m\}} \bigg|1-\frac{z}{\gamma_n} \bigg|
&  \prod_{n\in\Z\setminus \{0,l\}} \bigg|1-\frac{z}{n} \bigg|^{-1} = 
\sum_{n\in\Z\setminus\{0, m, l\}} 
\frac{z(\gamma_n - n)}{(n-z)\gamma_n} +O(1) \\
& = 
\sum_{n\in\Z\setminus\{0, m\}} 
\frac{m \delta_n}{(n-m)n} +O(1) 
= \sum_{n\in\Z\setminus\{0, m\}} \bigg(\frac{\delta_n}{n-m} - \frac{\delta_n}{n}\bigg) +O(1).  
\end{aligned}
$$
Applying the Abel transform and using \eqref{ab} it is easy to show that the series
$$
\sum_{n=m+1}^\infty \frac{\delta_n}{n-m} \qquad \text{and} \qquad  \sum_{n=-\infty}^{m-1} \frac{\delta_n}{n -m} 
$$
converge and their sums are uniformly bounded with respect to $m$. 
It is clear from our estimates that the product has a symmetric diagram $[-\pi i, \pi i]$ and, thus, the conjugate diagram
of $G$ equals $[0, 2\pi i]$. The conclusion of the lemma follows. 
\end{proof}

\begin{remark}
\label{glue}
{\rm  We will repeatedly use the following simple (and well-known) observation. 
Let $\Gamma \subset \Z$, $E = [0, a] \cup [b,c] \cup [d, 2\pi]$
and $\tilde E = [b,c] \cup [d, 2\pi +a]$, i.e., $\tilde E$ is obtained by gluing together the last interval of $E$
with its first interval shifted by $2\pi$.
Then the geometric properties (such as completeness or being a Riesz basis) 
of the system $\{e_\gamma\}_{\gamma \in \Gamma}$ are the same in the spaces $L^2(E)$ and $L^2(\tilde E)$. 
Indeed, the mapping 
$$
(\tilde Uf) (x) = 
\begin{cases} f(x), & x\in  [b,c] \cup [d, 2\pi], \\
f(x-2\pi), & x\in  (2\pi, 2\pi +a],
\end{cases}
$$
is a unitary operator from $L^2(E)$ to $L^2(\tilde E)$ and $U e_\gamma = e_\gamma$
for any $\gamma\in\Z$ due to periodicity. Of course, the same is true for any number of intervals.}
\end{remark}

\begin{corollary}
\label{bas2}
Let $E = [0, a] \cup [b, 2\pi]$. Then there exists a sequence $\Gamma = \{\gamma_n\} \subset \mathbb{Z}$ 
such that $\{e_\gamma\}_{\gamma\in \Gamma}$ is a Riesz basis in $L^2(E)$ and, moreover, there exists an entire function $G$ 
with the conjugate indicator diagram $[0, i|E|]$ such that its zero set coincides with $\Gamma$ and 
$|G(z)| \asymp {\rm dist}\,(z, \Gamma)$, $|\ima z| \le 1$.
\end{corollary}

\begin{proof}
Note that if $\Gamma \subset \Z$, then, by Remark \ref{glue}, $\{e_\gamma\}_{\gamma \in \Gamma}$ is a Riesz basis or not
simultaneously for the sets $E = [0, a] \cup [b, 2\pi]$ and $\tilde E = [b, 2\pi +a]$. Since $|\tilde E| = |E| <2\pi$,
we obtain the corresponding basis by an obvious rescaling of the system from Lemma \ref{bas1}. 
\end{proof}

In the case of three intervals a key ingredient of our construction is the following lemma from \cite{kn}. 

\begin{lemma} \textup(\cite[Lemma 2]{kn}\textup)
\label{key}
Let $S\subset [0, 2\pi]$. For $n=1, \dots, N$, denote by $A_n$ the set of those $x\in \big[0, \frac{2\pi}{N}\big]$ for which 
at least $n$ of the numbers $x +\frac{2\pi j}{N}$, $j=0, \dots, N-1$, belong to $S$. Assume
that $\Lambda_1, \dots, \Lambda_N \subset N\mathbb{Z}$ and $\{e_\lambda\}_{\lambda\in \Lambda_n}$ 
is a Riesz basis for $L^2(A_n)$. Then $\cup_{n=1}^N (\Lambda_n +n)$ is a Riesz basis for $L^2(S)$. 

Moreover, if $S$ is a union of $L$ disjoint intervals, then there exists $N\in \N$ such that each $A_n$ 
is a union of at most $L-1$ intervals. 
\end{lemma}

The last statement of the lemma should be understood in a ``cyclic way'' meaning that the union of intervals
$[0, a] \cup \big[ b, \frac{2\pi}{N} \big]$ should be considered as one interval $\big[ b, \frac{2\pi}{N} +a \big]$. 
If $\Lambda \subset N\Z$, then, by periodicity, $\{e_\lambda\}_{\lambda\in\Lambda}$ is a Riesz basis or not  simultaneously for 
$[0, a] \cup \big[ b, \frac{2\pi}{N} \big]$ and for $\big[ b, \frac{2\pi}{N} +a \big]$. 

\begin{proof}[Proof of Proposition \ref{basis}.] 
Note that if $\Gamma \subset \Z$, then $\{e_\gamma\}_{\gamma \in \Gamma}$ is a Riesz basis or not
simultaneously for the sets $E = [0, a] \cup [b,c] \cup [d, 2\pi]$ and $\tilde E = [b,c] \cup [d, 2\pi +a]$ (see Remark \ref{glue}). 
By Lemma \ref{key} we can choose $N$ such that all sets $A_1, \dots A_N$ constructed for the set $\tilde E$ 
are proper subintervals of $\big[0, \frac{2\pi}{N}\big]$ (understood in the cyclic sense). Rescaling and applying
Lemma \ref{bas1} we conclude that for each $n$ there exists a sequence $\Gamma_n \subset N\Z$ such that
$\{e_\gamma\}_{\gamma \in \Gamma_n}$ is a Riesz basis in $L^2(A_n)$ 
and the function 
$$
G_n(z) = z e^{i|A_n| z} {\rm v. p.} \prod_{\gamma\in \Gamma_n, \gamma\ne 0} \Big(1-\frac{z}{\gamma}\Big)
$$
has the conjugate diagram $[0, i |A_n|]$ and satisfies 
$|G_n(z)| \asymp  {\rm dist}\,(z, \Gamma_n)$, $|\ima z| \le 1$.
Then, by Lemma \ref{key},
$\Gamma = \cup_{n=1}^N (\Gamma_n+n) $ is a Riesz basis for $L^2(\tilde E)$
(and, hence, for $L^2(E)$) and the corresponding entire function
$G(z) = \prod_{n=1}^N G_n(z-n)$ will have the required properties.
\end{proof}

\begin{remark}
\label{compl}
{\rm It follows from the estimate $|G(z)| \asymp {\rm dist}\,(z, \Gamma)$, $|\ima z| \le 1$,
that $\Gamma $ is a uniqueness set for $PW_I$ for any interval $I$ with $|I| = |E|$. 
Indeed, if it is not the case, then 
there exists an entire function $H$ such that $HG \in PW_E $. Without loss of generality $I = [0, |E|]$. Then $H$ 
is of zero exponential type and is bounded (and even tend to zero) on the line $\ima z=1$, whence $H\equiv 0$.
In fact $\Gamma$ even generates a Riesz basis of reproducing kernels in $PW_I$, since $G$ is a sine-type function
(see \cite[Lecture 22]{lev} for details), but we do not use this fact.}
\end{remark}
\bigskip


\section{Structure of the biorthogonal system} 
\label{bio}

In this section we obtain a representation for the biorthogonal system 
to a complete and minimal system of reproducing kernels in $PW_E$. 
Let $E=\cup_{j=1}^N I^j$ be a finite union of disjoint intervals, let
$\{k_\lambda^E\}_{\lambda\in\Lambda}$ be a complete and minimal system of reproducing kernels in $PW_E$,
and let $\{f_\lambda\}_{\lambda\in\Lambda}$ be its biorthogonal system. 
Fix $N$ points $\lambda_1,....,\lambda_N \in \Lambda$ and put 
\begin{equation}
\label{dor}
F_k(z)=(z-\lambda_k) f_{\lambda_k}(z),\quad k=1,...,N.
\end{equation}
We can write $F_k = \sum_{j=1}^N F^j_k$, where $F_k^j$ has its spectrum in $I^j$. 

\begin{lemma}
\label{indep}
Let $E=\cup_{j=1}^N I^j$ be a finite union of intervals and let
$\{k_\lambda^E\}_{\lambda\in\Lambda}$ be a complete and minimal system of reproducing kernels in $PW_E$
such that $\Lambda$ satisfies \eqref{dens}.
Then there exist $\lambda_1,....,\lambda_N \in \Lambda$ such that the functions $F_1, \dots F_N$ 
defined by \eqref{dor} are linearly independent. 
\end{lemma}

\begin{proof}
We will discuss in detal the cases $N=2$ and $N=3$ which we will need for the proof of Theorem \ref{main}
and omit the proof for general $N$.
\medskip
\\
{\bf Case $N=2$.} Fix some $\lambda_1\in \Lambda$. If, for any $\lambda_2\in \Lambda\setminus \{\lambda_1\}$,
$F_1$ and $F_2$ are linearly dependent, then there exists a nonzero constant $c= c_{\lambda_1, \lambda_2}$ 
such that $F_2 = c_{\lambda_1, \lambda_2} F_1$. Then $\frac{F_1}{z-\lambda_2} \in PW_E$ whence, by Remark \ref{div2},
$F_1^1(\lambda_2) =0$ for any $\lambda_2\in \Lambda\setminus \{\lambda_1\}$. Since 
$F_1^1(z) = (z-\lambda_1) f_{\lambda_1}^1$ and $f_{\lambda_1}^1\in PW_{I^1}$, while
$\Lambda\setminus \{\lambda_1\}$ has upper density
$(|I^1| +|I^2|)/2\pi$, we conclude that $f_{\lambda_1}^1 \equiv 0$, whence $f_{\lambda}^1 \equiv 0$ 
for any $\lambda\in\Lambda$, an obvious contradiction.
\medskip
\\
{\bf Case $N=3$.} 
Assume that for any $\lambda_1, \lambda_2, \lambda_3 \in \Lambda$ there exist
$(\alpha, \beta, \gamma) \ne (0,0,0)$ such that $\alpha F_1+\beta F_2 +\gamma F_3 = 0$.
Then we have
$\frac{\alpha F_1+\beta F_2}{z-\lambda_3} \in PW_E$, whence, by Lemma \ref{div1}, 
$$
\alpha F^j_1(\lambda_3)+\beta F_2^j(\lambda_3) =0, \qquad j=1, 2, 3.
$$
We conclude that the function $F^1_1 F^2_2 - F^1_2 F^2_1$
vanish at $\lambda_3$. Recall that $\lambda_3$ is an arbitrary point in $\Lambda \setminus \{\lambda_1, \lambda_2\}$. 
The spectrum of $F^1_1 F^2_2 - F^1_2 F^2_1$ is an interval of length $|I^1| +
|I^2|$. Since $\Lambda$ has upper density $(|I^1| + |I^2| + |I^3|)/2\pi$, we have  
$F^1_1 F^2_2 - F^1_2 F^2_1 \equiv 0$. Analogously, 
$F^1_1 F^3_2 - F^1_2 F^3_1 = 
F^2_1 F^3_2 - F^2_2 F^3_1 \equiv 0$. Hence, for any $z$ (except a countable set), there exist $u(z)$ 
such that 
$$
(F_1^1(z), F_1^2(z), F_1^3(z)) = u(z) (F_2^1(z), F_2^2(z), F_2^3(z)). 
$$
Clearly, $u$ is a meromorhic function. 
Let us show that $u$ is constant. Otherwise, there exist $c\in \co$ such that
$F^j_1- c F^j_2$, $j=1,2,3$, have a common zero $\mu\notin\Lambda$ and $\frac{F_1- c F_2}{z-\mu}$ 
is a nonzero function in $PW_E$ vanishing on $\Lambda$, a contradiction. We conclude that
for any $\lambda_1, \lambda_2 \in \Lambda$ one has $F_1= cF_2$ for some nonzero constant $c$. 
Now we come to a contradiction in the same way as in the case $N=2$ above.
\end{proof}

We start with a representation of a system 
biorthogonal to a complete and minimal system of reproducing kernels in $PW_E$ when $E$  is a union of two intervals. 

\begin{proposition}
\label{bior2}
Let $E=I^1\cup I^2$ be a union of two disjoint intervals
and let $\{k_\lambda^E\}_{\lambda\in\Lambda}$ be a complete and minimal system of reproducing kernels in $PW_E$
such that $\Lambda$ satisfies \eqref{dens}.
Let $\lambda_1, \lambda_2\in \Lambda$  be such that the functions $F_1$ and $F_2$ are linearly independent. 
Consider the decomposition 
$$
F_1=F_1^1+F_1^2,  \qquad F_2=F_2^1+F_2^2,
$$
where $F_1^1, F_2^1 $ have their spectra in $I^1$ and $F_1^2, F_2^2 $ have their spectra in $I^2$.
Then there exist nonzero constants $c_\lambda$ such that
\begin{equation}
\label{telle}
f_\lambda(z)=c_\lambda \frac{F_1^1(\lambda) F_2(z) -F_2^1(\lambda) F_1(z)}{z-\lambda} 
= c_\lambda \frac{F_2^2(\lambda) F_1(z) - F_1^2(\lambda) F_2(z)}{z-\lambda}.
\end{equation}
\end{proposition}

\begin{proof}
By Lemma \ref{indep} we can find $\lambda_1, \lambda_2$ satisfying the hypothesis.
Put $\tilde f_\lambda(z) = \frac{F_1^1(\lambda) F_2(z) -F_2^1(\lambda) F_1(z)}{z-\lambda}$.
Since $\{f_\lambda\}_{\lambda\in\Lambda}$ is biorthogonal to $\{k_\lambda^E\}_{\lambda\in\Lambda}$,
we have $f_\lambda(\mu) = 0$ for $\mu\in\Lambda$, $\mu \ne \lambda$. Thus each of the functions $F_\lambda$ vanishes 
on $\Lambda$ and so $\tilde f_\lambda(\mu) = 0$ for $\mu\in\Lambda$,  $\mu \ne \lambda$.

Let us show that $\tilde f_\lambda \in PW_E$. We can decompose
$$
\begin{aligned}
F_1^1(\lambda) F_2(z)  - F_2^1(\lambda) F_1(z) = &
\big(F_1^1(\lambda) F^1_2(z)  -  F_2^1(\lambda) F^1_1(z) \big) \\
 + & \big(F_1^1(\lambda) F_2^2(z) - F_2^1(\lambda) F^2_1(z) \big).
\end{aligned}
$$
Obviously, $F_1^1(\lambda) F^1_2(z) - F_2^1(\lambda) F^1_1(z) $ (the ``projection'' of 
$F_1^1(\lambda) F_2(z) - F_2^1(\lambda) F_1(z) $ onto $I^1$) vanish at $\lambda$, and so $\tilde f_\lambda \in PW_E$
by Corollary \ref{div1}. Also $\tilde f_\lambda \perp k^E_\mu$, $\mu\in \Lambda\setminus\{\lambda\}$, 
in $PW_E$. It remains to show that $\tilde f_\lambda(\lambda) \ne 0$. Then, by uniqueness of the biorthogonal system, 
$f_\lambda = c_\lambda \tilde f_\lambda$ for some $c_\lambda\ne 0$. The second equality in \eqref{telle}
follows trivially since $F_1^1(\lambda) +F_1^2(\lambda) = F_2^1(\lambda) +F_2^2(\lambda) = 0$.

It remains to show that $\tilde  f_\lambda $ is not identically zero, whence by the uniqueness of the biorthogonal system
we conclude that $\tilde  f_\lambda$ is proportional to $f_\lambda$. Assume that $\tilde f_\lambda \equiv 0$. 
Since $F_1$ and $F_2$ are linearly independent,  we have $F_1^1(\lambda) = F_2^1(\lambda)=0$. 
Then, by Remark \ref{div2}, $\frac{F_1(z)}{z-\lambda}, \frac{F_2(z)}{z-\lambda} \in PW_E$ and are both orthogonal to 
$k^E_\mu$, $\mu\in \Lambda\setminus\{\lambda\}$. By uniqueness of the biorthogonal system, $F_1$ and $F_2$ are proportional,
a contradiction.
\end{proof}

\begin{corollary}
\label{fs}
In the notations of Proposition \ref{bior2} put
$$
S (z)= F_1^1(z) F_2(z) -  F_2^1(z) F_1(z)    = F_1^1(z) F^2_2(z) - F_1^2(z) F_2^1(z).
$$
Then the zero set of $S$ coincides with $\Lambda$ and all zeros of $S$ are simple. 
Moreover, we can find $t\in\R$ such that $e^{itz} S$ has the conjugate indicator diagram $[0, i|E|]$.
\end{corollary}

\begin{proof}
Assume that $S(w) =0$, $w\notin\Lambda$. Note that $S(z)$ is the determinant of the matrix with the 
rows $(F_1^1(z), F_1^2(z))$ and 
$(F_2^1(z), F_2^2(z))$ whence there exists a nontrivial pair $(\alpha, \beta)$ such that 
$\alpha F_1^1 + \beta F_2^1$ and $\alpha F_1^2 + \beta F_2^2$ vanish at $w$. By Remark \ref{div2}, 
$\frac{\alpha F_1 +\beta F_2}{z-w}$ belongs to $PW_E$ and vanish on $\Lambda$, whence $\alpha F_1 +\beta F_2 \equiv 0$,
a contradiction to linear independence.

If $\lambda\in \Lambda$ and $S$ has a zero at $\Lambda$ of order greater than $1$, then it is easy to see that $f_\lambda$ 
vanish at $\lambda$, a contradiction. 

Since $F_1^1, F_2^1$ have the spectra in $I^1$ and $F_1^2, F_2^2$ have the spectra in $I^2$, 
the conjugate diagram of $S$ is contained in some interval $[0, iR]$ with $R\le |E|$. By the density assumption \eqref{dens}
on $\Lambda$, it is a uniqueness set for any $PW_I$ with $|I| <|E|$. Then we conclude that $R=|E|$.
\end{proof}

Now we give a representation for the biorthogonal system for the general case of finite number of intervals.
The idea of the proof is similar, but the formulas become more involved. 
Recall that for $\lambda_1,....,\lambda_N \in \Lambda$  we put $F_k(z)=(z-\lambda_k) f_{\lambda_k}(z)$, $k=1,...,N$,
and we write $F_k = \sum_{j=1}^N F^j_k$, where $F_k^j$ has its spectrum in $I^j$. 

\begin{proposition}
\label{struct}
Let $E=\cup_{j=1}^N I^j$ be a finite union of intervals, let
$\{k_\lambda^E\}_{\lambda\in\Lambda}$ be a complete and minimal system of reproducing kernels in $PW_E$
satisfying \eqref{dens}, and let $\{f_\lambda\}_{\lambda\in\Lambda}$ be its biorthogonal system. 
Let $N$ points $\lambda_1,....,\lambda_N \in \Lambda$ be such that the system $\{F_k\}_{k=1}^N$ 
is linearly independent. Then

1. for any $\lambda\in \Lambda$ and $1\le l\le N$
there exists a constant $c_{\lambda, l} \ne 0$ such that
$$
f_\lambda(z)=\frac{c_{\lambda,l}}{z-\lambda}\sum_{k=1}^N a_k^l F_k(z),
$$
where 
$$
a_k^l= a_k^l(\lambda) = (-1)^{k+l} \det(F^i_j(\lambda))_{i \neq k, j\neq l};
$$ 

2. the entire function $S(z):=\det(F^j_k(z))_{1\leq k,j\leq N}$ has the zero set $\Lambda$ and all its zeros are simple.
\end{proposition}

\begin{proof}
1. Since $\sum_{j=1}^N F^j_k(\lambda)= F_k(\lambda) =0$, we conclude that
$\det(F^l_k(\lambda))_{1\leq k,l\leq N}=0$. 
Note that $a_k^l$  are cofactors of the elements $F_k^l(\lambda)$ and therefore
$$
\sum_{k=1}^N a_k^l  F^j_k(\lambda)=0, \quad j, l=1,...,N.
$$
Put 
$$
\tilde f_{\lambda, l}(z)=\frac{1}{z-\lambda}\sum_{k=1}^N a_k^l F_k(z).
$$
It follows from Remark \ref{div2} that $ \tilde f_{\lambda, l} \in PW_E$. On the other hand, 
$\tilde f_{\lambda, l}(\mu)=0$, $\mu\in\Lambda\setminus\{\lambda\}$. 
It remains to prove that $\tilde f_{\lambda,l}$ is not identically zero 
(then,  by the uniqueness of the biorthogonal system,
$\tilde f_{\lambda,l} = c_{\lambda, l} f_\lambda$ for some $c_{\lambda, l} \ne 0$).

In what follows we will use Proposition \ref{struct}  only in the case $N=3$.
Therefore, to simplify the presentation, we assume that $N=3$. The general case is analogous.

We fix $\lambda_1, \lambda_2$ and will look for $\lambda_3$. Assume that for any $\lambda_3 \in
\Lambda \setminus \{\lambda_1, \lambda_2\}$ there exist $\lambda\in \Lambda$ and $l$ such that
$\tilde f_{\lambda, l}(z) \equiv 0$. Since $F_1, F_2$ and $F_3$ are linearly independent, we have
$a_k^l(\lambda) =0$ for any $1\le k\le 3$. Without loss of generailty let $l=1$. Then 
$$
F_2^2(\lambda)F_3^3(\lambda) - F_2^3(\lambda)F_3^2(\lambda)=
F_1^2(\lambda)F_2^3(\lambda) - F_1^3(\lambda)F_2^2(\lambda)=
F_1^2(\lambda)F_3^3(\lambda) - F_1^3(\lambda)F_3^2(\lambda)=0.
$$
It follows that each two of the three vectors $(F_1^2(\lambda), F_1^3(\lambda))$,
$(F_2^2(\lambda), F_2^3(\lambda))$, $(F_3^2(\lambda), F_3^3(\lambda))$
are linearly dependent. Therefore there exist nontrivial $(\alpha, \beta)$ 
such that 
$$
\alpha F_2^2(\lambda) +\beta F_3^2(\lambda) = \alpha F_2^3(\lambda) +\beta F_3^3(\lambda)  =0,
$$
whence $\frac{\alpha F_2+\beta F_3}{z-\lambda} \in PW_E$. Similarly, 
there exist nontrivial $(\tilde \alpha, \tilde \beta)$ 
such that  $\frac{\tilde \alpha F_1+\tilde \beta F_2}{z-\lambda} \in PW_E$. Each of these functions 
vanish on $\Lambda\setminus\{\lambda\}$ and, by the uniqueness of the biorthogonal element,
they are proportional. We come to a contradiction with the linear independence of $F_1, F_2$ and $F_3$.
\medskip

2. First, we show that $S$ is not identically zero. 
Note that if $\{1,2,3\} \setminus \{k\} = \{i, j\}$, $\{1,2,3\} \setminus \{l\} = \{m, n\}$, where $i<j$, $m<n$,
then $a_k^l = (-1)^{k+l} (F_i^m(\lambda) F_j^n(\lambda) - F_i^n(\lambda) F_j^m(\lambda))$.
We introduce the following notation:
$$
F_{ij}^{mn} (z) = (-1)^{k+l} \big( F_i^m(z) F_j^n(z) - F_i^n(z) F_j^m(z) \big).
$$

Assume that $S\equiv 0$. Note that 
$$
S = F^{23}_{23} F_1+ F^{23}_{13} F_2 +F_{12}^{23} F_3.
$$
Fix some $\Lambda\in \Lambda$. With our notation, taking $l=1$, one has
$$
c (z-\lambda) f_\lambda(z)  =
F^{23}_{23}(\lambda)F_1(z)+ F^{23}_{13}(\lambda)F_2(z)+F_{12}^{23}(\lambda)F_3(z)
$$
for some $c\ne 0$. Then
$$
c (z-\lambda) f_\lambda(z) = 
(F^{23}_{23}(\lambda) - F^{23}_{23}(z) ) F_1(z)+ (F^{23}_{13}(\lambda) - F^{23}_{13}(z)) 
F_2(z)+(F_{12}^{23}(\lambda) - F_{12}^{23}(z)) F_3(z).
$$
Thus, $f_\lambda$ vanishes on $\lambda$, a contradiction. The same computations show that
$S$ cannot have a multiple zero at $\lambda\in\Lambda$.

Finally, if $S(w)=0$, $w\not\in\Lambda$, then there exists 
a non-trivial sequence $\{a_k\}^N_{k=1}$ such that 
$$
\sum_{k=1}^N a_k  F^j_k(w)=0, \quad j=1,...,N,
$$
and so the function $f_w(z)=\frac{1}{z-w}\sum_{k=1}^N a_kF_k(z)$ is in $PW_E$ 
and vanishes on $\Lambda$. Thus, $f_w\equiv 0$, again a contradiction 
with the linear independence of $F_1, F_2$ and $F_3$.
\end{proof}

\begin{corollary}
\label{upp}
Let $E=\cup_{j=1}^N I^j$ be a finite union of intervals and $\{e_\lambda\}_{\lambda\in\Lambda}$ be a complete and minimal system 
in $L^2(E)$. Then $\mathcal{D}_+(\Lambda) \le \frac{|E|}{2\pi}$.
\end{corollary}

\begin{proof} 
By Proposition \ref{struct}, the function $S(z)=\det(F^j_k(z))_{1\leq k,j\leq N}$ is nonzero
and vanishes on $\Lambda$. Recall that $F^j_k$ has its spectrum in $I^j$ and its conjugate indicator diagram
is contained  in $i I^j$. Hence, the conjugate indicator diagram of $S$ is contained in an interval of length $|E|$
and the estimate follows.
\end{proof}
\bigskip


\section{Proof for the case of two intervals}
\label{case2}

In this section we prove completeness of the biorthogonal system for 
the case when $E=I^1\cup I^2 = [0, a] \cup [b, 2\pi]$. We include a separate proof for this case
since it is much simpler and more elementary than the three intervals case. 
In particular, we do not need to use the key lemma from \cite{kn}.

Let $\{k_\lambda^E\}_{\lambda\in\Lambda}$ be a complete and minimal system of reproducing kernels in $PW_E$
and let $\{f_\lambda\}_{\lambda\in\Lambda}$ be its biorthogonal system. 
Without loss of generality we assume that $\Lambda\cap \Z = \emptyset$ (otherwise we can simply shift $\Lambda$
by a real constant). Assume that $h\in PW_E$ is orthogonal to 
$\{f_\lambda\}_{\lambda\in\Lambda}$. 

By Corollary \ref{bas2} there exists a Riesz basis $\{k^E_\gamma\}_{\gamma \in \Gamma}$ in $PW_E$
such that $\Gamma = \{\gamma _n\} \subset \Z$ and its generating function $G$ satisfies
$|G(z)| \asymp  {\rm dist}\,(z, \Gamma)$, $|\ima z| \le 1$. We expand $h$ with respect to $\{k^E_\gamma\}_{\gamma \in \Gamma}$,
$$
h = \sum_n \bar a_n k^E_{\gamma_n}, \qquad \{a_n\}\in \ell^2.
$$
Hence, using the first representation for $f_\lambda$ from \eqref{telle}, we get 
$$
\sum_n a_n  \frac{F_2^1(\lambda) F_1(\gamma_n) - F_1^1(\lambda) F_2(\gamma_n)}{\gamma_n-\lambda} = (f_\lambda, h) =0, \qquad \lambda\in \Lambda.
$$
Put 
$$
L(z) = G(z)\sum_n a_n  \frac{F_2^1(z) F_1(\gamma_n) - F_1^1(z) F_2(\gamma_n)}{z- \gamma_n}.
$$
Since $L$ vanishes on $\Lambda$ we can write $L= e^{itz} S V^1$ for some entire function $V^1$, 
where $S$ and $t$ are defined in Corollary \ref{fs}. 

Recall that the conjugate indicator diagrams of $G$ and $e^{itz} S$ coincide with $[0, i|E|]$. Since the spectra of $F_1^1, F_2^1$ 
are contained in $I^1$, the diagram of $L$ is obviously contained
in the (Minkowski) sum $i I^1 + [0, i|E|]$. Hence, the diagram of $V^1$ is contained in $i I^1$.

We can also use the second representation for $f_\lambda$ from \eqref{telle}. Then we conclude that
$$
G(z)\sum_n a_n  \frac{F_1^2(z) F_2(\gamma_n) - F_2^2(z) F_1(\gamma_n)}{z- \gamma_n} = e^{itz} SV^2
$$
for another entire function $V^2$. Arguing as above we conclude that the diagram of $V^2$ is contained in $i I^2$. 

Now note that $L(\gamma_n) = a_n G'(\gamma_n) S(\gamma_n)$, whence $V^1(\gamma_n) = e^{-it\gamma_n} a_n G'(\gamma_n)$.
Since, by construction, $|G'(\gamma_n)| \asymp 1$, we conclude that $\{V^1(\gamma_n)\} \in\ell^2$. 
Since $\Gamma$ generates a Riesz basis on the interval $[b, 2\pi +a]$ of the length $|I^1| + |I^2| = |E|$
and the width of the diagram of $V^1$ equals $|I^1|$, it follows from a variant 
of the classical Cartwright theorem (see, e.g., \cite[Section 10.5]{boas})
that $V^1\in PW_{I^1}$. 

Using the fact that $ F_1^2(z) F_2(z) - F_2^2(z) F_1(z) = S(z)$, we analogously get 
$V^2(\gamma_n) = e^{-it\gamma_n} a_n G'(\gamma_n)$, whence $V^2\in PW_{I^2}$. Now, consider the function
$$
W(z) = e^{2\pi i z} V^1(z) - V^2(z).
$$
Since $\Gamma\subset \Z$, we have $e^{2\pi i \gamma_n}=1$ and so $W$ vanish on $\Gamma$. 
Also, $W$ belongs to $PW_{[b, 2\pi +a]}$. Since $\Gamma$ is a uniqueness set for 
$PW_{[b, 2\pi +a]}$ (see Remark \ref{compl}) 
we conclude that $W\equiv 0$. Since the functions $e^{2\pi i z} V^1$ and $V^2$ have disjoint spectra, 
we have $V^1=V^2\equiv 0$ and, finally, $a_n\equiv0$ and $h=0$.
\qed
\bigskip


\section{Proof for the case of three intervals}
\label{case3}

Let $E$ be a union of three disjoint intervals,  $E=\cup_{j=1}^3 I^j$. Without loss of generality 
$$
E=[0,a]\cup[b,c]\cup[d,2\pi].
$$

We can assume that $\Lambda\cap\mathbb{Z}=\emptyset$ 
(otherwise consider the sequence $\Lambda+\delta$, $\delta\in\mathbb{R}$).
By Lemma \ref{basis} there exists a Riesz basis $\{e_\gamma\}_{\gamma\in\Gamma}$ 
such that $\Gamma\subset \Z$ and there exists an entire function $G$ with conjugate indicator diagram $[0, i|E|]$ and 
with simple zeros exactly at $\Gamma$,
$$
G(z)=e^{isz}\, {\rm v.p}\,\prod_{\gamma\in\Gamma}\biggl{(}1-\frac{z}{\gamma}\biggr{)},
$$
and $|G'(\gamma)|\asymp 1$, $\gamma\in\Gamma$.
\medskip

As usual, we consider the equivalent problem about the system $\{k^E_\lambda\}_{\lambda\in\Lambda}$ 
of reproducing kernels in $PW_E$. Denote by $\{f_\lambda\}_{\lambda\in\Lambda}$ the system biorthogonal 
to  $\{k^E_\lambda\}_{\lambda\in\Lambda}$. 
Assume that $\{f_\lambda\}_{\lambda\in\Lambda}$ is not complete. Then there exists a non-trivial 
$h\in PW_E$ which is orthogonal to 
$\{f_\lambda\}_{\lambda\in\Lambda}$. Consider the expansion of $h$ with respect to the Riesz basis 
$\{k^E_\gamma\}_{\gamma\in\Gamma}$,
$$
h(z)=\sum_n \bar a_n k^E_{\gamma_n}(z),\qquad \{a_n\}\in\ell^2.
$$

Let $\lambda_1, \lambda_2, \lambda_3\in\Lambda$ 
be as in Proposition \ref{struct}. Put $F_k(z) = (z-\lambda_k) f_{\lambda_k}(z)$, 
$k=1,2,3$, and write $F_k = F_k^1 + F_k^2+F_k^3$, 
where $F_k^j$ has its spectrum in $I^j$. 
By Proposition \ref{struct}, for any $l=1,2,3$ and $a_k^l=(-1)^{k+l} \det(F_i^j(\lambda))_{i \neq k, j\neq l}$
one has
$$
\frac{a_1^l F_1(z)+ a_2^l F_2(z)+ a_3^l F_3(z)}{z-\lambda} = c f_{\lambda}(z)
$$
for some nonzero constant $c = c_{\lambda, l}$. We use the notation
introduced in the proof of Proposition \ref{struct}:
for $\{1,2,3\} \setminus \{k\} = \{i, j\}$, $\{1,2,3\} \setminus \{l\} = \{m, n\}$, where $i<j$, $m<n$,
we put 
$$
F_{ij}^{mn} (z) = (-1)^{k+l} \big( F_i^m(z) F_j^n(z) - F_i^n(z) F_j^m(z) \big).
$$

Let us first take $l=1$. Then, with the above notation, we have
$$
\frac{F^{23}_{23}(\lambda)F_1(z)+ F^{23}_{13}(\lambda)F_2(z)+F_{12}^{23}(\lambda)F_3(z)}{z-\lambda}
= c_\lambda f_{\lambda}(z).
$$
Hence, since $h$ is orthogonal to $f_\lambda$,
$$
0=(c_\lambda f_\lambda, h)=\sum_na_n\frac{F^{23}_{23}(\lambda)F_1(\gamma_n)+
F^{23}_{13}(\lambda)F_2(\gamma_n)+F_{12}^{23}(\lambda)F_3(\gamma_n)}{\gamma_n-\lambda}.
$$
Consider the entire function 
\begin{equation}
\label{Leq}
L(z)=G(z)\sum_n a_n\frac{F^{23}_{23}(z) F_1(\gamma_n)+
F^{23}_{13}(z)F_2(\gamma_n)+F_{12}^{23}(z)F_3(\gamma_n)}{z-\gamma_n}.
\end{equation}
Then $L$ vanish on $\Lambda$. 

Recall that the function $S(z)=\det(F^j_k(z))_{1\leq k,j\leq 3}$ from Proposition \ref{struct} 
vanish exactly on $\Lambda$. Since the spectra of the functions $F^j_k$ are contained in $I^j$, we have $S\in 
\mathcal{P}\cdot PW_I$, $I = I^1+I^2+I^3$, thus, the spectrum of $S$ (in the distributional sense) 
is contained in $I$. We can choose $t\in \R$ such that $e^{itz} S$ has the spectrum $[0, |E|]$.
Now we can write
\begin{equation}
\label{Leq2}
L(z)=e^{itz} S(z)V^{23}(z)
\end{equation}
for some entire function $V^{23}$. 

Note that, by trivial linear algebra, 
$$
F^{23}_{23}(z)F_1(z)+ F^{23}_{13}(z)F_2(z)+F_{12}^{23}(z)F_3(z) = 
\det(F^j_k(z))_{1\leq k,j\leq 3} = S(z).
$$
Now, comparing the values of the left-hand and right-hand parts in \eqref{Leq} 
at $\gamma_n$ and using the fact that $S(\gamma_n) \ne 0$, we get
$$
V^{23}(\gamma_n)=e^{-it\gamma_n} a_n G'(\gamma_n).
$$
Note that the conjugate indicator diagrams of the functions $F^{23}_{23},F^{23}_{13},F^{23}_{12}$ 
are contained in $i(I^2+I^3)$, 
while $\frac{G(z)}{z-\gamma_n}\in PW_{[0, |E|]}$. Hence, the conjugate diagram of 
$L$ is contained in $i(I^2+I^3+[0, |E|])$. Since the diagram of $e^{itz} S$ 
equals $i[0, |E|]$, we conclude that
$$
\diag V^{23}\subset i(I^2+I^3)
$$
(recall that we denote by $\diag f$ the conjugate indicator diagram of an entire function $f$ of exponential type).
Since $V^{23}\bigl{|}_\Gamma \in\ell^2(\Gamma)$ and $\Gamma$ has density $\frac{|E|}{2\pi}>\frac{|I^2+I^3|}{2\pi}$,
we conclude that $V^{23}\in PW_{I^2+I^3}$ by the classical Cartwright theorem \cite[Section 10.5]{boas}.

If we consider the coefficients $a_k^l$ with $l=2$ or $l=3$, we get two 
more alternative representations for functions $c_\lambda f_\lambda$ (with different constants $c_\lambda$).
E.g., if we start with $a_1^2 = F^{13}_{23}(\lambda)$, $a_2^2 = F^{13}_{13}(\lambda)$, $a_3^2 = F^{13}_{12}(\lambda)$,
then we obtain an entire function $V^{13} \in PW_{I^1+I^3}$. Similarly, starting from $\{a_k^3\}$ we get
$V^{12}\in PW_{I^1+I^2}$. All these three functions coincide at $\Gamma$:
$$
V^{23}(\gamma_n)=V^{13}(\gamma_n)=V^{12}(\gamma_n)= e^{-it\gamma_n}a_nG'(\gamma_n).
$$ 

Note that 
$$
I^2+I^3=[b+d,c+2\pi], \qquad I^1+I^2=[b,a+c], \qquad I^1+I^3=[d,a+2\pi].
$$
Consider the function $e^{2\pi i z}  V^{12}(z) -V^{23}(z)$, which vanish on $\Gamma$.
Its conjugate diagram is contained in the interval $i [b+d,a+c+2\pi]$, whose length equals $|E|$,
and so $$e^{2\pi i z}  V^{12}(z) -V^{23}(z) \in PW_{ [b+d,a+c+2\pi]}.$$ 
Since $\Gamma$ is a uniqueness set for $PW_I$ for any interval $I$ with $|I| = |E|$ (see Remark \ref{compl}), 
we conclude that $ e^{2 \pi i z} V^{12}(z) \equiv V^{23}(z)$. Since
$$
\diag V^{23}\subset i [b+d,c+2\pi], \qquad \diag e^{2\pi i z} V^{12} \subset i [b+2\pi,a+c+2\pi],
$$
it follows that $\diag V^{23}\subset i [b+2\pi,c+2\pi]$, $\diag V^{12}\subset i [b,c]$. 

Next, consider the function $V^{13}(z)-V^{12}(z)$
which also vanish on $\Gamma$. Its spectrum is contained in 
$\tilde E=[b,c]\cup[d, a+2\pi]$, and so  $V^{13}-V^{12} \in PW_{\tilde E}$. 
Consider $\{e^{i\gamma t}\}_{\gamma\in\Gamma}$ 
as a system in $L^2(\tilde{E})$. Since $\Gamma\subset{\mathbb{Z}}$ and 
$\{e^{i\gamma t}\}_{\gamma\in\Gamma}$ is a Riesz basis in $L^2(E)$, 
it is also a Riesz basis in $L^2(\tilde{E})$ (by Remark \ref{glue}, moving a part of a set by $2\pi$ does not change 
the geometry of exponentials with integer frequencies). 
Therefore, $\{k^{\tilde E}_\gamma\}_{\gamma\in \Gamma}$ is a Riesz basis in $PW_{\tilde E}$.
We conclude that 
$V^{12}(z)\equiv V^{13}(z)$. Since, $\diag V^{12}\cap\diag V^{13}=\emptyset$ 
we conclude that $V^{12}=V^{13}=V^{23}=0$. Hence, $a_n\equiv0$ and, thus, $h=0$.
\qed

\end{document}